\documentclass[runningheads]{llncs}
\usepackage{amsmath, amssymb, lineno, lscape,hyperref,subfig}
\usepackage[T1]{fontenc}
\usepackage{graphicx}
\usepackage{color}

\def\control{0}

\newcommand\blfootnote[1]{%
  \begingroup
  \renewcommand\thefootnote{}\footnote{#1}%
  \addtocounter{footnote}{-1}%
  \endgroup
}

\begin{document}
\title{How accurate does Newton have to be?}
\author{Carl Christian Kjelgaard Mikkelsen \inst{1}\orcidID{0000-0002-9158-1941 } \and
Lori{\'e}n L{\'o}pez-Villellas \inst{2}\orcidID{0000-0002-1891-4359} \and Pablo Garc{\'i}a-Risue{\~no} \inst{3} \orcidID{0000-0002-8142-9196}}
\authorrunning{C. C. Kjelgaard Mikkelsen et al. }
\institute{
  Department of Computing Science, Ume{\aa} University, 90187 Ume{\aa}, Sweden \email{spock@cs.umu.se} \\
  \and Barcelona Supercomputing Center, Barcelona, Spain, \\
  \email{lorien.lopez@bsc.es} \\
  \and Independent scholar \\
  \email{risueno@unizar.es, garcia.risueno@gmail.com}
}

\maketitle              % typeset the header of the contribution
\begin{abstract} We analyze the convergence of quasi-Newton methods in exact and finite precision arithmetic. In particular, we derive an upper bound for the stagnation level and we show that any sufficiently exact quasi-Newton method will converge quadratically until stagnation. In the absence of sufficient accuracy, we are likely to retain rapid linear convergence. We confirm our analysis by computing square roots and solving bond constraint equations in the context of molecular dynamics. We briefly discuss implications for parallel solvers.
\end{abstract}

\blfootnote{Preprint accepted by the conference PPAM-2022 on 2022-06-30. To appear in Springer's LNCS, vol. 13826.}

\section{Introduction} Let $\Omega \subseteq \mathbb{R}^n$ be open, let $F \in C^1(\Omega, \mathbb{R}^n)$ and consider the problem of solving
\begin{equation*}
F(x) = 0.
\end{equation*}
If the Jacobian $F'$ of $F$ is nonsingular, then Newton's method is given by
\begin{equation} \label{equ:newton} % REF
  x_{k+1} = x_k - s_k, \quad F'(x_k)s_k = F(x_k).
\end{equation}
A quasi-Newton method is any iteration of the form
\begin{equation} \label{equ:quasi-newton} % REF 
  y_{k+1} = y_k - t_k, \quad F'(y_k)t_k \approx F(y_k).
\end{equation}
In exact arithmetic, we expect local quadratic convergence from Newton's method \cite{mysovskii1949}. Quasi-Newton methods normally converge locally and at least linearly and some methods, such as the secant method, have superlinear convergence \cite{ortega_iterative_1970,kelley_iterative_1995}. In finite precision arithmetic, we cannot expect convergence in the strict mathematical sense and we must settle for stagnation near a zero \cite{tisseur2001}. In this paper we analyze the convergence of quasi-Newton methods in exact and finite precision arithmetic. In particular, we derive an upper bound for the stagnation level and we show that any sufficiently exact quasi-Newton method will converge quadratically until stagnation. We confirm our analysis by computing square roots and solving bond constraint equations in the context of molecular dynamics.

\section{Auxiliary results}

%% A sequence $\{x_k\}_{k = 1}^\infty \subset \mathbb{R}^{n \times n}$ is said to converge to $x \in \mathbb{R}^{n \times n}$ with order $p \ge 1$ and rate $\mu > 0$ if
%% \begin{equation}
%%   \frac{\|x- x_{k+1}\|}{\|x - x_k\|^p} \rightarrow \mu, \quad k \rightarrow \infty, \quad k \in \mathbb{N}.
%% \end{equation}

%% In the case of linear convergence, i.e., $p=1$ we require $\mu < 1$. Superlinear convergence corresponds to $p>1$ and quadratic convergence corresponds to $p=2$. Let $x \in \mathbb{R}^n$ and $y \in \mathbb{R}^n$.

The line segment $l(x,y)$ between $x$ and $y$ is defined as follows:
\begin{equation*}
  l(x,y) = \{ tx + (1-t) y \: : \: t \in [0,1] \}.
\end{equation*}

The following lemma is a standard result that bounds the difference between $F(x)$ and $F(y)$ if the line segment $l(x,y)$ is contained in the domain of $F$. 
  
\begin{lemma} \label{lem:mean-value-theorem} Let $\Omega \subseteq \mathbb{R}^n$ be open and let $F \in C^1(\Omega,\mathbb{R}^n)$. If $l(x,y) \subset \Omega$, then
  \begin{equation*}
    F(x) - F(y) = \int_0^1 F'(tx + (1-t)y)(x-y) dt
  \end{equation*}
  and
  \begin{equation*}
    \|F(x) - F(y)\| \leq M \|x-y\|.
  \end{equation*}
  where
  \begin{equation*}
    M = \sup \{ \| F'(tx + (1-t)y) \| \: : \: t \in [0,1] \}.
  \end{equation*}
\end{lemma}

\ifnum \control=1
\begin{proof} Let $\phi : [0,1] \rightarrow \mathbb{R}^n$ denote the function given by
  \begin{equation*}
    \phi(t) = F(tx + (1-t)y).
  \end{equation*}
  Then $\phi$ is well-defined and the chain rule implies that $\phi \in C^1([0,1],\mathbb{R}^n)$ with
  \begin{equation*}
    \phi'(t) = F'(tx + (1-t)y)(x-y).
  \end{equation*}
  By the fundamental theorem of calculus we have
  \begin{equation*}
    F(x) - F(y) = \phi(1) - \phi(0) = \int_0^1  F'(tx + (1-t)y)(x-y) dt.
  \end{equation*}
  By the triangle inequality we have
   \begin{equation*}
     \| F(x) - F(y) \| \leq \int_0^1  \| F'(tx + (1-t)y) \| \| x-y \| dt \leq M \|x-y\|.
    \end{equation*}
   This completes the proof.
\end{proof}
\fi

It is convenient to phrase Newton's method as the functional iteration:
\begin{equation*}
  x_{k+1} = g(x_k), \quad g(x) = x - F'(x)^{-1}F(x).
\end{equation*}
and to express the analysis of quasi-Newton methods in terms of the function $g$. The next lemma can be used to establish local quadratic convergence of Newton's method.

\begin{lemma} \label{lem:newton:error-formula} Let $\Omega \subseteq \mathbb{R}^n$ be open and let $F \in C^1(\Omega, \mathbb{R}^n)$. Let $z$ denote a zero of $F$ and let $x \in \Omega$. If $F'(x)$ is nonsingular and if $l(x,z) \subset \Omega$, then
  \begin{equation*}  
  g(x) - z = C(x)(x-z)
  \end{equation*}
  where
  \begin{equation*}
    C(x) = F'(x)^{-1} \left( \int_0^1 \left[ F'(x) - F'(tx + (1-t)z) \right] dt \right) 
  \end{equation*}
  Moreover, if $F'$ is Lipschitz continuous with Lipschitz constant $L>0$, then
  \begin{equation*}
    \|g(x) - z \| \leq \frac{1}{2} \|F'(x)^{-1}\| L \|x - z \|^2.
  \end{equation*}
\end{lemma}

\ifnum\control=1
\begin{proof} The function $g$ is defined at $x \in \Omega$ because $F'(x)$ is nonsingular. We have
  
  \begin{align}
    g(x) - z
    &= (x - z) - F'(x)^{-1} ( F(x) - F(z)) \\
    &= F'(x)^{-1} \left[ F'(x) (x-z) - \left( F(x) - F(z) \right) \right] \\
    &= F'(x)^{-1} \int_0^1 \left [ F'(x) - F'(tx + (1-t)z) \right] (x-z) dt \\
    &= C(x) (x-z).
  \end{align}
  
  Now if $F'$ is Lipschitz continuous with Lipschitz constant $L>0$, then
  \begin{equation*}
    \|C(x)\| \leq \|F'(x)^{-1} \| \int_0^1 L (1-t) \|x-z\| dt = \frac{1}{2} \|F'(x)^{-1} \| L  \|x-z\|.
  \end{equation*}
  This completes the proof.
\end{proof}
\fi

The following lemma allows us to write any approximation as a very simple function of the target vector.

\begin{lemma} \label{lem:error-function} Let $x \in \mathbb{R}^n$ be nonzero, let $y \in \mathbb{R}^n$ be an approximation of $x$ and let $E \in \mathbb{R}^{n \times n}$ be given by
  \begin{equation*}
    E = \frac{1}{x^Tx} (y-x)x ^T.
  \end{equation*}
  Then
  \begin{equation*}
    y = (I + E) x, \quad \|E\| = O\left(\frac{\|x-y\|}{\|x\|} \right), \quad y \rightarrow x, \quad y \not = x.
  \end{equation*}
  In the special case of the 2-norm we have
  \begin{equation*}
    \|E\|_2 = \frac{\|x-y\|_2}{\|x\|_2}.
  \end{equation*}
\end{lemma}

\begin{proof} It is straightforward to verify that
  \begin{equation*}
    (I + E)x = x + \frac{1}{x^Tx} (y-x) x^T x = x + (y-x) = y.
  \end{equation*}
  Moreover, if $z$ is any vector, then
  \begin{equation*}
  \|Ez\| \leq \frac{1}{\|x\|_2^2} \|y-x\| \|x^Tz\| = \left( \frac{\|x^T\| \|x\|}{\|x\|_2^2} \right) \left( \frac{\|x-y\|}{\|x\|} \right) \|z\|.
  \end{equation*}
  In the case of the 2-norm, we have
  \begin{equation*}
    \|Ez\|_2 \leq \frac{\|x-y\|_2}{\|x\|_2} \|z\|_2
  \end{equation*}
  for all $z \not = 0$ and equality holds for $z=x$. This completes the proof.
\end{proof}

\section{Main Results} \label{sec:main-results}

In the presence of rounding errors, \emph{any} quasi-Newton method can written as
\begin{equation} \label{equ:quasi-newton-with-errors:1}
  x_{k+1} = (I + D_k) \Big( x_k - (I + E_k) F'(x_k)^{-1} F(x_k) \Big).
\end{equation}
Here $D_k \in \mathbb{R}^{n \times n}$ is a diagonal matrix which represents the rounding error in the subtraction and $E_k \in \mathbb{R}^{n \times n}$ measures the difference between the computed correction and the correction used by Newton's method. We simply treat the update $t_k$ needed for the quasi-Newton method \eqref{equ:quasi-newton} as an approximation of the update $s_k = F'(x_k)^{-1} F(x_k)$ needed for Newton's method \eqref{equ:newton} and define $E_k$ using Lemma \ref{lem:error-function}. It is practical to restate iteration \eqref{equ:quasi-newton-with-errors:1} in terms of the function $g$, i.e.,
\begin{equation} \label{equ:quasi-newton-with-errors:2} % REF
  x_{k+1} = (I + D_k) \Big( g(x_k)  - E_k F'(x_k)^{-1} F(x_k) \Big).
\end{equation}

We shall now analyze the behavior of iteration \eqref{equ:quasi-newton-with-errors:2}. For the sake of simplicity, we will assume that there exist nonnegative numbers $K$, $L$, and $M$ such that
\begin{equation*}
\forall x \: : \: \|F'(x)^{-1}\| \leq K, \quad  \|F'(x) - F'(y) \| \leq L \|x-y\|, \quad \|F'(x)\| \leq M.
\end{equation*}
In reality, we only require that these inequalities are satisfied in a neighborhood of a zero. We have the following generalization of Lemma \ref{lem:newton:error-formula}.

\begin{theorem} \label{thm:main} The functional iteration given by equation \eqref{equ:quasi-newton-with-errors:2} satisfies
  \begin{multline} \label{equ:generalized-error-formula} % REF
    x_{k+1} - z  = g(x_k) - z - E_k F'(x_k)^{-1} F(x_k) \\ + D_k \big[ g(x_k) - E_k F'(x_k)^{-1} F(x_k) \big]
  \end{multline}
and
\begin{multline} \label{equ:generalized-upper-bound} % REF
    \|x_{k+1} - z \| \leq \frac{1} {2} L K \|x_k - z \|^2 + \|E_k\| K M \|x_k-z\| \\ + \|D_k\| \left( \|z\| + \frac{1} {2} L K \|x_k - z\|^2 + \|E_k\| K M \|x_k-z\| \right).
  \end{multline}
\end{theorem}

\begin{proof} It is straightforward to verify that equation \eqref{equ:generalized-error-formula} is correct. Inequality \eqref{equ:generalized-upper-bound} follows from equation  \eqref{equ:generalized-error-formula} using the triangle inequality, Lemma \ref{lem:mean-value-theorem}, and Lemma \ref{lem:newton:error-formula}. The second occurrence of the term $\|g(x_k)\|$ can be bounded using the inequality
  \begin{equation*}
    \|g(x_k)\| \leq \|z\| + \|g(x_k) - z\|.
  \end{equation*}
  This completes the proof.
\end{proof}

It is practical to focus on the case of $z \not = 0$ and restate inequality \eqref{equ:generalized-upper-bound} as
\begin{equation} \label{equ:generalized-upper-bound:relative-error} % REF
  r_{k+1} \leq  \frac{1} {2} L K (1 +  \|D_k\|) \|z\| r_k^2 + \|E_k\| K M  (1 +  \|D_k\|) r_k \\ + \|D_k\|
\end{equation}
where $r_k$ is the normwise relative forward error given by
\begin{equation*}
  r_k = \|z - x_k\|/\|z\|.
\end{equation*}

\subsection{Stagnation} \label{sec:stagnation}

We assume that the sequences $\{D_k\}$ and $\{E_k\}$ are bounded. Let $D$ and $E$ be nonnegative numbers that satisfy
\begin{equation} \label{equ:upper-bounds-D-and-E} % REF
  \|D_k\| \leq D, \quad \|E_k \| \leq E.
\end{equation}
In this case, inequality \eqref{equ:generalized-upper-bound:relative-error} implies
\begin{equation*}
  r_{k+1} \leq \frac{1}{2} LK(1+D) \|z\| r_k^2 + EMK(1+D) r_k  + D.
\end{equation*}
It is certain that the error will be reduced, i.e., $r_{k+1} < r_k$ when
\begin{align*}
  D 
  &< r_k - \left( \frac{1}{2} LK(1+D) \|z\| r_k^2 + EMK(1+D) r_k^2 \right)  \\
  &= \left( 1 - EMK(1+D) \right) r_k -  \frac{1}{2} LK(1+D) \|z\| r_k^2.
\end{align*}
This condition is equivalent to the following inequality:
\begin{equation*}
   D  - \left[ 1 - EMK(1+D) \right] r_k +  \frac{1}{2} LK(1+D)\|z\| r_k^2 < 0.
\end{equation*}
This is an inequality of the second degree. The roots are
\begin{equation*}
  \lambda_{\pm} = \frac{\left[ 1 - EMK(1+D) \right] \pm \sqrt{\left[ 1 - EMK(1+D) \right]^2 - 2 LK(1+D)D\|z\|}}{LK(1+D)\|z\|}.
\end{equation*}
If $D$ and $E$ are sufficiently small then the roots are positive real numbers and the error will certainly be reduced provided
\begin{equation*}
  \lambda_{-} < r_k < \lambda_{+}.
\end{equation*}
It follows that we cannot expect to do better than
\begin{equation*}
  r_k = \frac{\|z-x_k\|}{\|z\|} \approx \lambda_{-}.
\end{equation*}
If $D$ and $E$ are sufficiently small, then a Taylor expansion ensures that
\begin{equation*}
  \lambda_{-} \approx \frac{D}{\left(1-EMK(1+D)\right)^2}
\end{equation*}
is a good approximation. We cannot expect to do better than $r_{k+1} = \lambda_-$, but the threshold of stagnation is not particularly sensitive to the size of $E$.

\subsection{The Decay of the Error} \label{sec:decay}

We assume that the sequences $\{D_k\}$ and $\{E_k\}$ are bounded. Let $D$ and $E$ be upper bounds that satisfy \eqref{equ:upper-bounds-D-and-E}. Suppose that we are not near the threshold of stagnation in the sense that
\begin{equation} \label{equ:far-from-stagnation:linear-case} % REF
  D \leq C r_k.
\end{equation}
for a (modest) constant $C>0$. In this case, inequality \eqref{equ:generalized-upper-bound:relative-error} implies
\begin{equation} \label{equ:local-linear-decay} % REF
  r_{k+1} \leq \rho_k r_k, \quad \rho_k = \frac{1} {2} L K (1 + D) \|z\| r_k + E K M  (1 +  D) + C.
\end{equation}
If $C<1$, then we may have $\rho_k < 1$, when $r_k$ and $E$ are sufficiently small. This explains when and why local linear decay is possible.
We now strengthen our assumptions. Suppose that there is a $\lambda \in (0,1]$ and $C_1 > 0$ such that
\begin{equation} \label{equ:error-decay:lambda} % REF
  \| E_k \| \leq C_1 r_k^{\lambda}
\end{equation}
and that we are far from the threshold of stagnation in the sense that
\begin{equation} \label{equ:far-from-stagnation:lambda} % REF
 D \leq C_2 r_k^{1 + \lambda}
\end{equation}
for a (modest) constant $C_2 > 0$. In this case, inequality \eqref{equ:generalized-upper-bound:relative-error} implies 
\begin{equation} \label{equ:superlinear-decay} % REF
   r_{k+1} \leq  \left[ \frac{1} {2} L K (1 + D) \|z\| r_k^{1-\lambda} + C_1 K M  (1 +  D) + C_2 \right] r_k^{1 + \lambda}.
\end{equation}
This explains when and why local superlinear decay is possible.

\subsection{Convergence} \label{sec:convergence}

We cannot expect a quasi-Newton method to converge unless the subtraction $y_{k+1} = y_k - t_k$ is exact. Then $D_k = 0$ and inequality \eqref{equ:generalized-upper-bound:relative-error} implies 
\begin{equation*}
  r_{k+1} \leq \eta_k r_k, \quad \eta_k = \left( \frac{1} {2} L K \|z\| r_k + \|E_k\| K M \right).
\end{equation*}
We may have $\eta_k < 1$ for all $k$, provided $E = \sup \|E_k\|$ and $r_0$ are sufficiently small. This explains when and why local linear convergence is possible. We now strengthen our assumptions. Suppose that there is a $\lambda \in (0,1]$ and a $C>0$ such that 
\begin{equation*}
  \forall k \in \mathbb{N} \: : \: \|E_k\| \leq C r_k^{\lambda}.
\end{equation*}
In this case, inequality \eqref{equ:generalized-upper-bound:relative-error} implies 
\begin{equation*}
  r_{k+1} \leq \left( \frac{1} {2} L K \|z\| r_k^{1-\lambda} + C K M \right)  r_k^{1 + \lambda}.
\end{equation*}
This inequality allows us to establish local convergence of order at least $1+\lambda$.

\subsection{How Accurate Does Newton Have To Be?} \label{sec:how-accurate}

% Our objective is to answer what is essentially the titular question of this paper, i.e., what accuracy is required to achieve quadratic convergence when computing the corrections needed for Newton's method.
We will assume the use of normal IEEE floating point numbers and we will apply the analysis given in Section \ref{sec:decay}. If we use the 1-norm, the 2-norm or the $\infty$-norm, then we may choose $D=u$, where $u$ is the unit roundoff. Suppose that equations \eqref{equ:error-decay:lambda} and \eqref{equ:far-from-stagnation:lambda} are satisfied with $\lambda = 1$. Then inequality \eqref{equ:superlinear-decay} reduces to
\begin{equation*}
   r_{k+1} \leq  \left[ \frac{1} {2} L K (1 + u) \|z\| + C_1 K M  (1 +  u) + C_2 \right] r_k^2.
\end{equation*}
Due to the basic limitations of IEEE floating point arithmetic we cannot expect to do better than
\begin{equation*}
  r_{k+1} = O(u), \quad u \rightarrow 0, \quad u > 0.
\end{equation*}
It follows that we \emph{never} need to do better than
\begin{equation*}
  \|E_k\| = O(\sqrt{u}), \quad u \rightarrow 0, \quad u > 0.
\end{equation*}

\section{Numerical Experiments}

\subsection{Computing square roots} \label{sec:square-roots}

Let $\alpha>0$ and consider the problem of solving the nonlinear equation
\begin{equation*}
  f(x) = x^2 - \alpha = 0
\end{equation*}
with respect to $x>0$ using Newton's method. Let $r_k$ denote the relative error after $k$ Newton steps. A simple calculation based on Lemma \ref{lem:newton:error-formula} yields
\begin{equation*}
  |r_{k+1}| \leq r_k^2/2, \quad |r_k| \leq 2 \left( |r_0|/2 \right)^{2^k}.
\end{equation*}
We see that convergence is certain when $|r_0| < 2$. The general case of $\alpha > 0$ can be reduced to the special case of $\alpha \in [1,4)$ by accessing and manipulating the binary representation directly. Let $x_0 : [1,4] \rightarrow \mathbb{R}$ denote the best uniform linear approximation of the square root function on the interval $[1,4]$. Then
  \begin{equation*}
    x_0(\alpha) = \alpha/3 + 17/24, \quad  |r_0(\alpha)|  \leq 1/24.
  \end{equation*}
  In order to illustrate Theorem \ref{thm:main} we execute the iteration
  \begin{equation*}
    x_{k+1} = x_k - (1+e_k)f(x_k)/f'(x_k)
  \end{equation*}
  where $e_k$ is a randomly generated number. Specifically, given $\epsilon > 0$ we choose $e_k$ such that $|e_k|$ is uniformly distributed in the interval $[\frac{1}{2} \epsilon, \epsilon]$ and the sign of $e_k$ is positive or negative with equal probability. Three choices, namely $\epsilon = 10^{-2}$ (left), $\epsilon = 10^{-8}$ (center) and $\epsilon = 10^{-12}$ (right) are illustrated in Figure \ref{fig:combined}.
  
   In each case, eventually the perturbed iteration reproduces either the computer's internal representation of the square root or stagnates with a relative error that is essentially the unit roundoff $u=2^{-53} \approx 10^{-16}$. When $\epsilon = 10^{-2}$ the quadratic convergence is lost, but the relative error is decreased by a factor of approximately $\epsilon = 10^{-2}$ from one iteration to the next, i.e., extremely rapid linear convergence. Quadratic convergence is restored when $\epsilon$ is reduced to $\epsilon = 10^{-8} \approx \sqrt{u}$. Further reductions of $\epsilon$ have no effect on the convergence as demonstrated by the case of $\epsilon = 10^{-12}$. We shall now explain exactly how far this experiment supports the theory that is presented in this paper.
  \paragraph{Stagnation}
  By Section \ref{sec:stagnation} we expect that the level of stagnation is essentially independent of the size of $E$, the upper bound on the relative error between the computed step and the step needed for Newton's method. This is clearly confirmed by the experiment.
  \paragraph{Error Decay}
  Since we are always very close to the positive zero of $f(x) = x^2 - \alpha$ we may choose
  \begin{equation*}
    L \approx 2, \quad K|z| \approx 1/2, \quad MK \approx 1, 
  \end{equation*}
  In the case of $\epsilon = 10^{-2}$, Figure~\ref{fig:combined} (left) shows that we satisfy inequality \eqref{equ:far-from-stagnation:linear-case} with $D = u$ and $C = \epsilon < 1$, i.e.,
  \begin{equation*}
    u \leq \epsilon r_k, \quad 0 \leq k < 5.
  \end{equation*}
  By equation \eqref{equ:local-linear-decay} we must have
  \begin{equation*}
    r_{k+1} \leq \rho_k r_k, \quad \rho_k \approx 2 \epsilon, \quad 0 < k < 5.
  \end{equation*}
  This is exactly the linear convergence that we have observed. In the case of $\epsilon = 10^{-8}$, Figure~\ref{fig:combined} (center) shows that we satisfy inequality \eqref{equ:far-from-stagnation:lambda} with $C_2 = 1$ and $\lambda = 1$, i.e.,
  \begin{equation*}
    u \leq r_k^2, \quad k = 0,1.
  \end{equation*}
  By inequality \eqref{equ:superlinear-decay} we must have quadratic decay in the sense that
  \begin{equation*}
    r_{k+1} \leq C r_k^2, \quad C \approx \frac{3}{2}, \quad k = 0,1.
  \end{equation*}
  Manual inspection of Figure \ref{fig:combined} reveals that the actual constant is close to $1$ and certainly smaller than $C \approx \frac{3}{2}$. By Section \ref{sec:how-accurate} we do not expect any benefits from using an $\epsilon$ that is substantially smaller than $\sqrt{u}$. This is also supported by the experiment.

\subsection{Constrained molecular dynamics} \label{sec:md}

The objective is to solve a system of differential algebraic equations
\begin{align*}
  q'(t) &= v(t), \\
  M v'(t) &= f(q(t)) - g'(q(t))^T \lambda(t), \\
  g(q(t)) &=0.
\end{align*}
Here $q$ and $v$ are vectors that represent the position and velocity of all atoms, $M$ is a nonsingular diagonal mass matrix, $f$ represents the external forces acting on the atoms and $-g'(q)^T \lambda$ represents the constraint forces. Here $g'$ is the Jacobian of the constraint function $g$. The standard algorithm for this problem is the SHAKE algorithm \cite{shake1977}. It uses a pair of staggered uniform grids and takes the form
\begin{align}
  v_{n+1/2} &= v_{n-1/2} + h M^{-1} \left( f(q_n) - g'(q_n)^T \lambda_n \right), \nonumber \\
  q_{n+1} &= q_n + h v_{n + 1/2}, \nonumber \\
  g(q_{n+1}) &= 0, \label{equ:nonlinear:1}
\end{align}
where $h>0$ is the fixed time step and $q_n \approx q(t_n)$, $v_{n+\frac{1}{2}} \approx v(t_{n+\frac{1}{2}})$, where $t_n = nh$ and $t_{n+\frac{1}{2}} = (n+1/2)h$. Equation \eqref{equ:nonlinear:1} is really a nonlinear equation for the unknown Lagrange multiplier $\lambda_n$, specifically
\begin{equation*}
  g(\phi_n(\lambda)) = 0, \quad \phi_n(\lambda)= q_n + h(v_{n-\frac{1}{2}} + h M^{-1}( f(q_n) - g'(q_n)^T \lambda)).
\end{equation*}
The relevant Jacobian is the matrix 
\begin{equation*}
  A_n(\lambda) = \left( g(\phi_n(\lambda)) \right)' =  g'(\phi_n(\lambda))M^{-1} g'(q_n)^T.
\end{equation*}
The matrix $A_n(\lambda)$ is close to the constant symmetric matrix $S_n$ given by
\begin{equation*}
  S_n = g'(q_n)M^{-1} g'(q_n)^T
\end{equation*}
simply because $\phi_n(\lambda) = q_n + O(h)$ as $h \rightarrow 0$ and $h>0$. It is therefore natural to investigate if the constant matrix $S_n^{-1}$ is a good approximation of $A_n^{-1}(\lambda)$. 

For this experiment, we executed a production molecular dynamics run using the GROMACS \cite{berendsen1995} package. We replaced the constraint solver used by GROMACS's SHAKE function with a quasi-Newton method based on the matrix $S_n$.  Our experiment was based on GROMACS's Lysozyme in Water Tutorial \cite{lysozyme-tutorial}. We simulated a hen egg white lysozyme \cite{lysozyme2022} molecule submerged in water inside a cubic box. Lysozyme is a protein that consists of a single polypeptide chain of 129 amino acid residues cross-lined at 4 places by disulfide bonds between cysteine side-chains in different parts of the molecule. Lysozyme has $1960$ atoms and $1984$ bond length constraints.
Before executing the production run, we added ions to the system to make it electrically neutral. The energy of the system was minimized using the steepest descent algorithm until the maximum force of the system was below 1000.0 kJ/(mol$\cdot$nm). Then, we executed 100 ps of a temperature equilibration step using a V-Rescale thermostat in an NVT ensemble to stabilize the temperature of the system at 310 K. To finish, we stabilized the pressure of the system at 1 Bar for another 100 ps using a V-Rescale thermostat and a Parrinello-Rahman barostat in an NPT ensemble.
We executed a 100 ps production run with a 2 fs time step using an NPT ensemble with a V-Rescale thermostat and a Parrinello-Rahman barostat with time constants of 0.1 and 2 ps, respectively. We collected the results of the constraint solver every 5 ps starting at time-step 5 ps, for a total of 20 sample points.
Specifically, we recorded the normwise relative error $r_k = \|\lambda_n-x_k\|_2/\|\lambda_n\|_2$ as a function of the number $k$ of quasi-Newton steps using the symmetric matrix $S_n$ instead of the nonsymmetric matrix $A_n$ and we recorded $\|E_k\|_2 = \|s_k - t_k\|_2/\|s_k\|_2$ where $t_k$ is needed for a quasi-Newton step and $s_k$ is needed a Newton step. By \eqref{equ:local-linear-decay} we have $r_{k+1} \leq \rho_k r_k$, but we cannot hope for more than $r_{k+1} \approx \rho_k r_k$ where $\rho_k = O(\|E_k\|_2)$ and this is indeed what we find in the Figure \ref{fig:validation} until we hit the level of stagnation where the impact of rounding errors is keenly felt.

\section{Related Work}

It is well-known that Newton's method has local quadratic convergence subject to certain regularity conditions. The simplest proof known to us is due to Mysovskii \cite{mysovskii1949}. Dembo et. al. \cite{dembo1982inexact} analyzed the convergence of quasi-Newton methods in terms of the ratio between the norm of linear residual, i.e., $r_k = F(x_k) - F'(x_k)t_k$ and the norm of the nonlinear residual $F(x_k)$. Tisseur \cite{tisseur2001} studied the impact of rounding errors in terms of the backward error associated with approximating the Jacobians and computing the corrections, as well as the errors associated with computing the residuals. Here we have pursued a third option by viewing the correction $t_k$ as an approximation of the correction $s_k$ needed for an exact Newton step. Tisseur found that Newton's method stagnate at a level that is essentially independent of the stability of the solver and we have confirmed that this is true for quasi-Newton methods in general. It is clear to us from reading Theorem 3.1 of Dennis and Moore's paper \cite{dennis1977} that they would instantly recognize Lemma 3, but we cannot find the result stated explicitly anywhere. Forsgren \cite{forsgren2009} uses a stationary method for solving linear systems to construct a quasi-Newton method that is so exact that the convergence is quadratic. Section \ref{sec:square-roots} contains a simple illustration of this phenomenon.

\section{Conclusions}

Quasi-Newton methods can also be analyzed in terms of the relative error between Newton's correction and the computed correction. We achieve quadratic convergence when this error is $O(\sqrt{u})$. This fact represent an opportunity for improving the time-to-solution for nonlinear equations. General purpose libraries for solving sparse linear systems apply pivoting for the sake of numerical accuracy and stability. In the context of quasi-Newton methods we do not need maximum accuracy. Rather, there is some freedom to pivot for the sake of parallelism. If we fail to achieve quadratic convergence, then we are likely to still converge rapidly. It is therefore worthwhile to develop sparse solvers that pivot mainly for the sake of parallelism.

\subsubsection{Acknowledgments}

Prof. I. Argyros commented on an early draft of this paper and provided the reference to the work of I. P. Mysovskii. The first author is supported by eSSENCE, a collaborative e-Science programme funded by the Swedish Research Council within the framework of the strategic research areas designated by the Swedish Government. This work has been partially supported by the Spanish Ministry of Science and Innovation (contract PID2019-107255GB-C21/AEI/10.13039/501100011033), by the Generalitat de Catalunya (contract 2017-SGR-1328), and by Lenovo-BSC Contract-Framework Contract (2020).


%\bibliographystyle{splncs04}
%\bibliography{references}

 \begin{landscape}
  \begin{figure}[ht]
    \includegraphics[width=18cm]{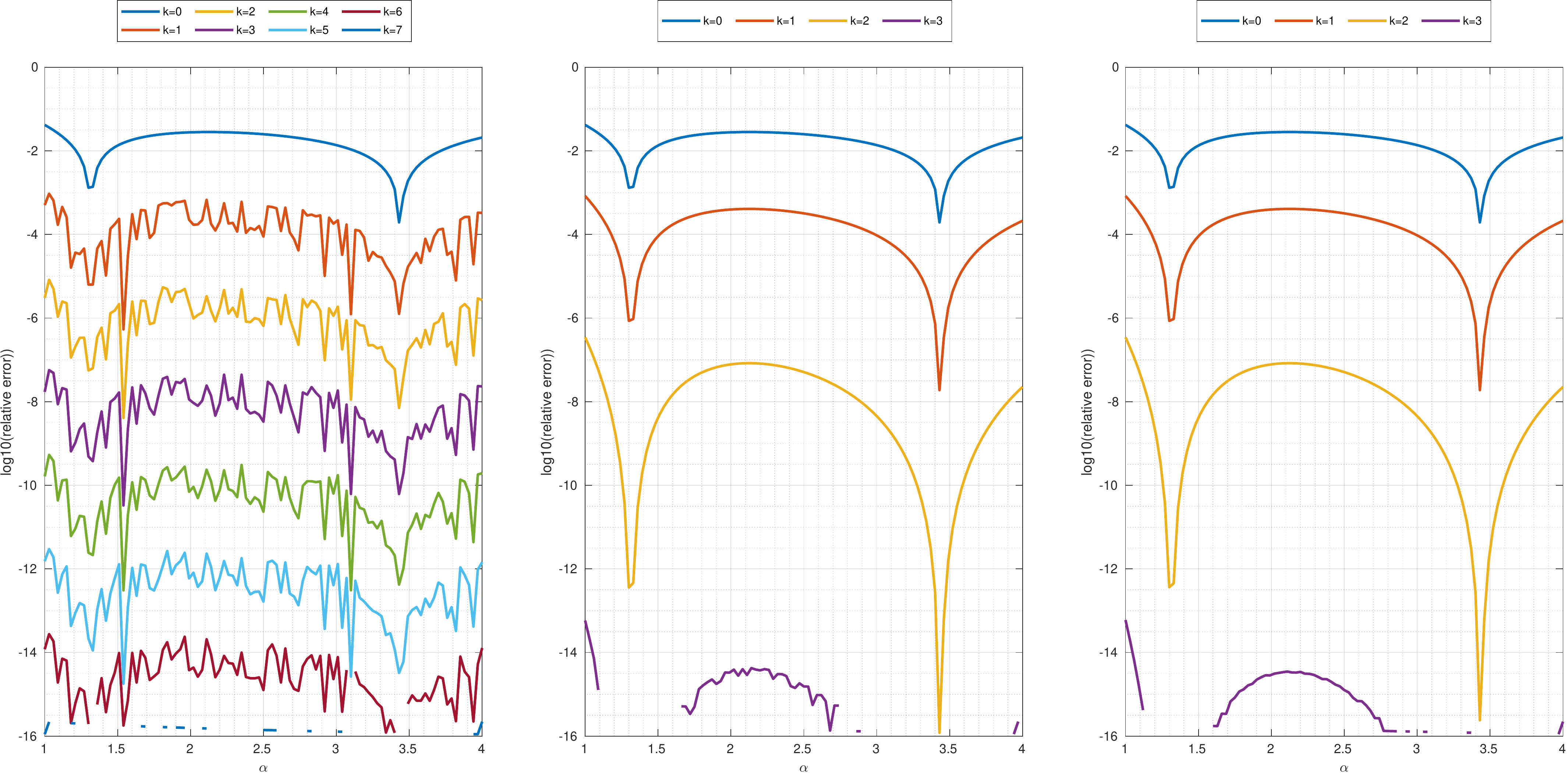}
\caption{The impact of inaccuracies on the convergence of Newton's method for a computing square roots. Newton's corrections have been perturbed with random relative errors of size $\epsilon \approx 10^{-2}$ (left), $\epsilon \approx 10^{-8}$ (center) and $\epsilon \approx 10^{-12}$. In each case, the last iteration produces an approximation that matches the computer's value of the square root at many sample points. In such cases, the computed relative error is $0$. Therefore, it is not possible to plot a data point and the last curve of each plot are discontinuous. } \label{fig:combined}
\end{figure}
\end{landscape}

\begin{landscape}
  \begin{figure}[ht]
    \subfloat[Constraint violation\label{fig:constraint-violation}]{%
      \includegraphics[width=6cm]{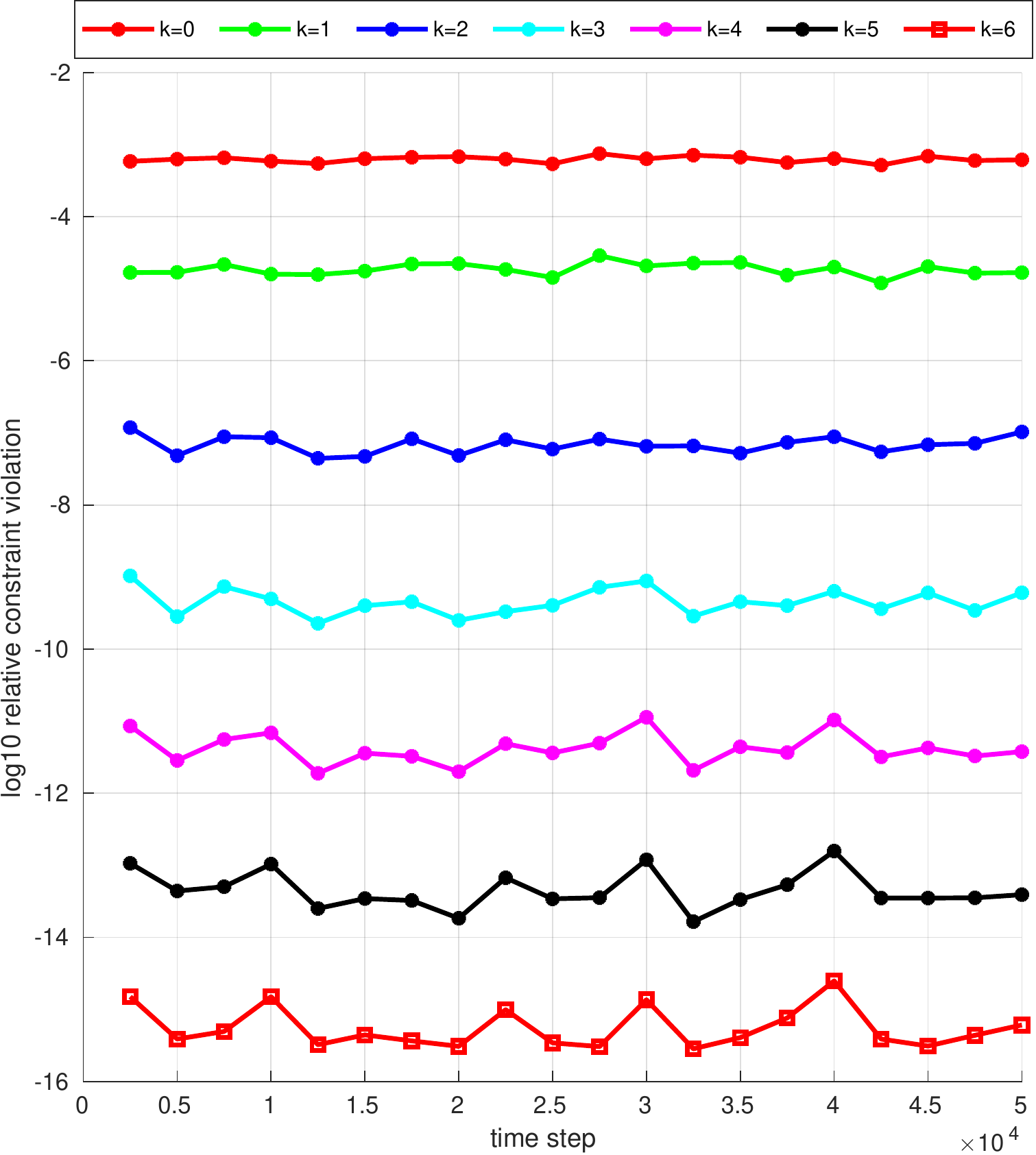}%
    }\hfil
    \subfloat[Relative error\label{fig:relative-error}]{%
      \includegraphics[width=6cm]{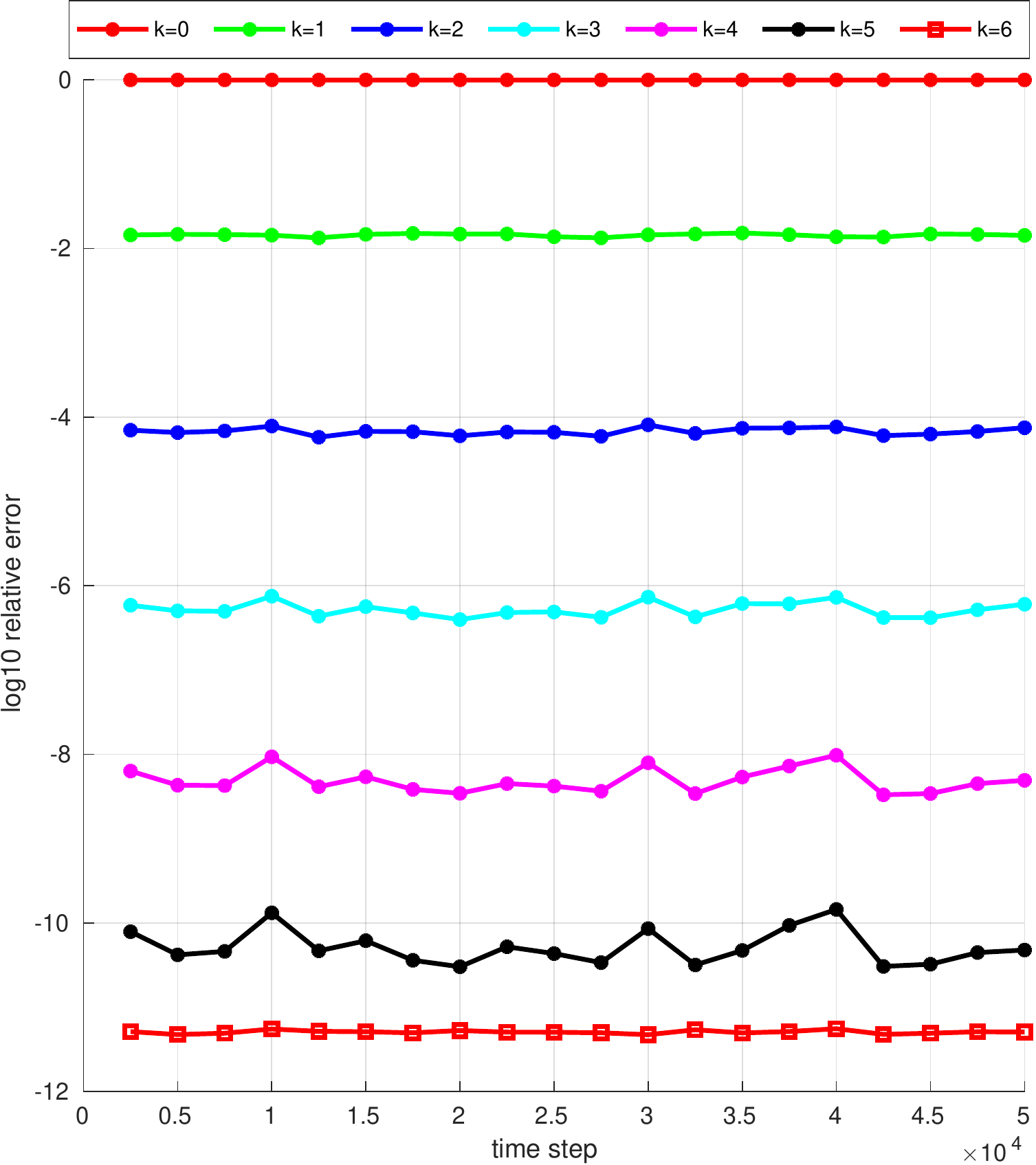}%
    }\hfil
   \subfloat[Validation\label{fig:validation}]{%
      \includegraphics[width=6cm]{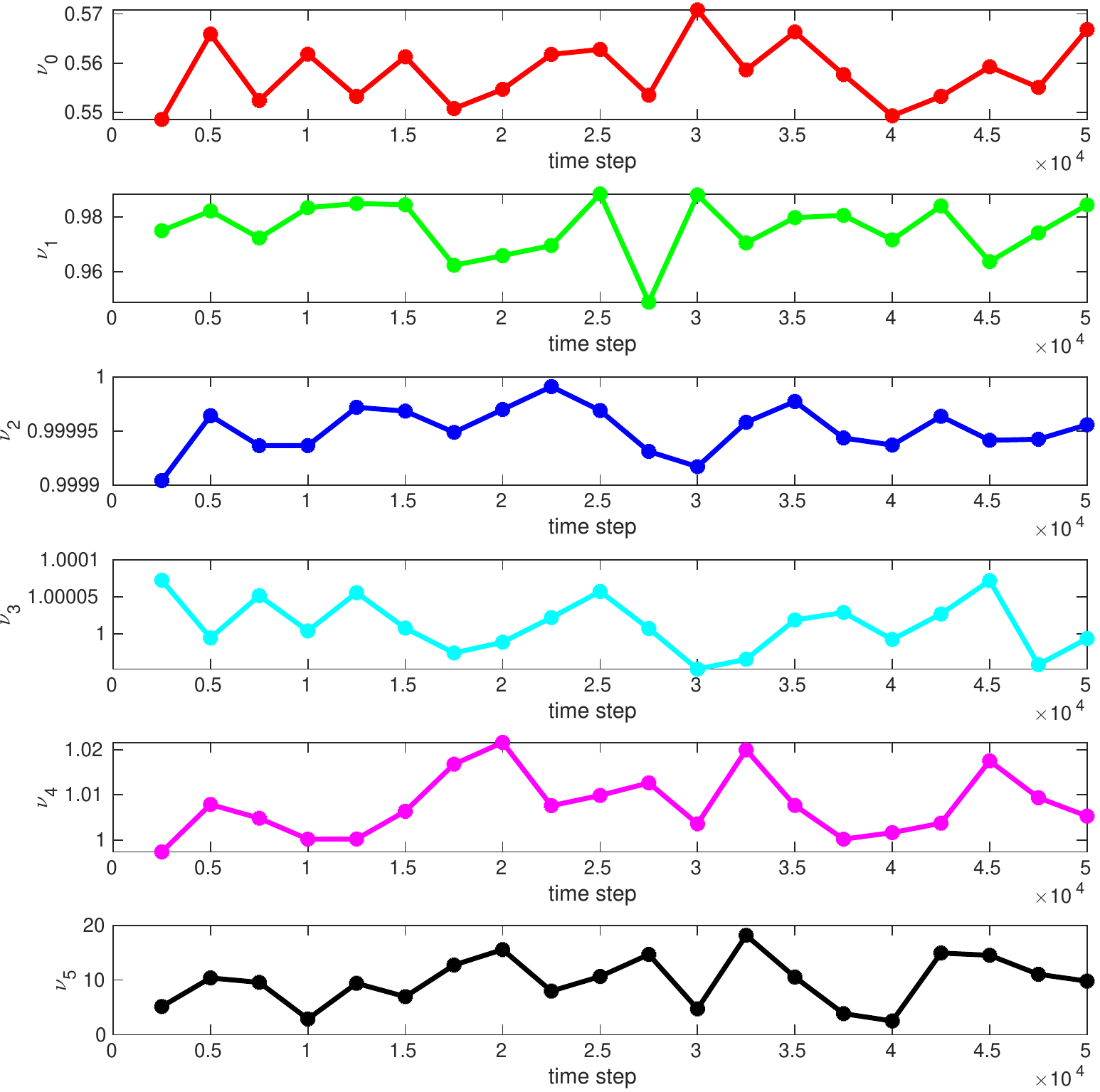}%
    }
   \caption{Data generated during a simulation of lysozyme in water using GROMACS. The GROMACS solver have been replaced with a quasi-Newton method that uses a fixed symmetric approximation of the Jacobian. Figure~\ref{fig:constraint-violation} is mainly of interest to computational chemists. It shows that the maximum relative constraint violation always stagnates at a level that is essentially the IEEE double precision unit roundoff after 6 quasi-Newton steps. The convergence is always linear and the rate of convergence is $\mu \approx 10^{-2}$. Figure~\ref{fig:relative-error} shows the development of the relative error $r_k$ between the relevant zero $z$, i.e., the Lagrange multiplier for the current time step and the approximations generated by $k$ steps of the quasi-Newton method. The convergence is always linear and the rate of convergence is $\mu \approx 10^{-2}$. Figure~\ref{fig:validation} provides partial validation of a theoretical result. Specifically, the fractions $\nu_k = r_{k+1}/(r_k \|E_k\|_2)$ are plotted for $k=0,1,2,3,4,5$. When $\nu_k$ is modest, we have experimental verification that the rate of convergence is essentially $\|E_k\|$.}
  \label{fig:molecular-dynamics}
  \end{figure}
\end{landscape}
 
\end{document}